\documentclass[12pt,a4paper, notitlepage]{article}
\usepackage[a4paper] {geometry}
\usepackage{amsmath,amsfonts,amssymb,mathrsfs,amsthm}
\usepackage{multirow}
\usepackage[english]{babel}
\usepackage[all]{xy}
\usepackage{tikz}
\numberwithin{equation}{section}
\numberwithin{figure}{section}
\newtheorem{thm}{Theorem}[section]
\newtheorem{prop}[thm]{Proposition}

\newtheorem{coro}[thm]{Corollary}
\theoremstyle{definition}
\newtheorem{defx}[thm]{Definition}
\newtheorem{rem}[thm]{Remark}

\newcommand{\R}{I\!\!R}

\DeclareMathSizes{12}{13}{10}{10}

\title{Cohomogeneity One Groupoid Analysis of the Dynamical System of Rings of Continuous Functions}
\author{{\emph{N. O. Okeke \;  Email: nickokeke@dui.edu.ng, }}
\\ Physical and Mathematical Sciences, Dominican University, Ibadan \\
\emph{M. E. Egwe \;  murphy.egwe@ui.edu.ng}
\\ Department of Mathematics, University of Ibadan, Ibadan, Nigeria}

\date{\vspace{-5ex}}

\begin{document}
\maketitle

\begin{abstract}
Using the group $G(1)$ of invertible elements and the maximal ideals $\mathfrak{m}_x$ of the commutative algebra $C(X)$ of real-valued functions on a compact regular space $X$, we define a Borel action of the algebra on the measure space $(X,\mu)$ with $\mu$ a Radon measure. The zero sets $Z(X)$ of the algebra $C(X)$ is used to study the ergodicity of the $G(1)$-action via its action on the maximal ideals $\mathfrak{m}_x$ which defines an action groupoid $\mathcal{G} = \mathfrak{m}_x \ltimes G(1)$ trivialized on $X$. The resulting measure groupoid $(\mathcal{G},\mathcal{C})$ is used to define a proper action on the generalized space $\mathcal{M}(X)$. The existence of slice at each point of $\mathcal{M}(X)$ present it as a cohomogeneity-one $\mathcal{G}$-space. The dynamical system of the algebra $C(X)$ is defined by the action of the measure groupoid $(\mathcal{G},\mathcal{C}) \times \mathcal{M}(X) \to \mathcal{M}(X)$.
\end{abstract}
\let\thefootnote\relax
\let\thefootnote\relax

\section*{Introduction}
We start with the phase space constituted by the parametrized algebraic closed (geometric) points $\{\mathfrak{m}_x\}_{x \in X}$, the maximal ideals, of the commutative algebra $C(X)$ of continuous real-valued functions on a completely regular compact space $X$, the time evolution of the system is given by the continuous transformation $\varphi : X \to X$ of the compact space. The dynamical symmetries is given as the invariance of the system under the transformations of coordinates or parameters identifying its configurations, or the commutation of action of the symmetry (group) transformations with time evolution of the system. According to F. Strocchi \cite{Strocchi2008}, the transformation $\varphi : x \to \varphi(x)$ defines a symmetry if \\
(a) it induces an invertible mapping of configurations $\varphi_* : \mathfrak{m}_x \to \varphi_*(\mathfrak{m}_x) \equiv \mathfrak{m}_{\varphi(x)}$; \\ (b) it leaves the dynamical behaviour of the system invariant \[\alpha^t\varphi(\mathfrak{m}_x) = \alpha^t\mathfrak{m}_{\varphi(x)} \equiv \mathfrak{m}_{\varphi(x)(t)} = \mathfrak{m}_{\varphi(x(t))} = \varphi(\alpha^t\mathfrak{m}_x).\]
These are related to the concept of ergodicity as average over time as given by the iteration of transformations, related to average over space given by invariance of measures, and the measurable functions representing the observables. In the case of the action of a compact Lie group on a smooth manifold, the fixed condition of invariance of Haar measure gave rise to the cohomological equation, the solution of which afforded the harmonic analysis of the discrete dynamical system \cite{OkekeEgwe2017}. We replace the smooth manifold with the generalized space $\mathcal{M}(X)$ of Radon measures on $X$ in order to understand the nature of the symmetry preserved in the orbits.

We employ the cohomogeneity one groupoid $\mathcal{G}$-space analysis introduced in \cite{OkekeEgwe2018} to (measure) groupoid actions on the generalized space $\mathcal{M}(X)$ to achieve our goal. Both Peter Hahn \cite{Hahn78} and Arlan Ramsay \cite{Ramsay71} have used measure groupoid to understand Mackey's virtual groups and its associated measure theory. Our focus is more on the implication of the slice theorem on the measure groupoid approach to (von Neumann algebras) measure theory. The proper actions connect directly to slice theorem used in the cohomogeneity-one $\mathcal{G}$-space analysis. The main results of the paper are Theorems 3.2, 4.3, 5.6, 7.1 and Proposition 6.4.

\section{Algebraic and Topological Structures of $C(X)$}
According to Atiyah and Macdonald \cite{AtiyahDonald69}, prime ideals (which are modules) play a central role in commutative algebras because of their generalization of the idea of prime numbers in number theory and points in geometry. Further, the focusing or restriction of attention to neighbourhood of a point is subsequently captured by the process of localization at a prime ideal. The paper is on $C(X)$-the commutative algebra (lattice) of real-valued continuous functions on a compact space $X$. We will therefore focus our attention on the maximal ideals, in reformulation of the algebra into localized action groupoid. We recommend \cite{AtiyahDonald69} for basic understanding of a commutative ring or algebra. From the characterization of commutative rings or algebra we highlight the following facts connecting them immediately to groupoid.

The first is that by definition of an $R$-module $M$ of a ring $R$, the map $R \to End(M)$ is a representation of $R$ on the space of linear transformations of $M$. The second is that given a ring $R$ (or an algebra $\mathcal{A}$), by definition, the units of the ring $R$ form a group $G(1)$ with an action $G(1) \times M \to M$ on every $R$-module $M$. The action gives us an action groupoid or a transformation groupoid $\mathcal{G} = G(1) \ltimes M$ over each module $M$. Thus, $\mathcal{G} = G(1) \times M \rightrightarrows M$ is an action groupoid. We will now highlight the algebraic and topological structures of the commutative algebra $C(X)$ which will help us understand the structure of $X$ from the action groupoid trivialized on it.

According to Gillman and Jerison \cite{Gillman1960}, given a compact space $X$, every real valued function on $X$ is continuous if $X$ is discrete. So $\R^X = C(X)$. Conversely, whenever $\R^X = C(X)$ then the characteristic function of every set in $X$ is continuous, showing that $X$ is discrete. This equivalence between the continuity of a real valued functions on $X$ and the discreteness of $X$ is the foundation for using the algebraic variety (or the zero-sets) of $C(X)$ to characterize a regular compact space $X$ in \cite{Gillman1960}.
\begin{defx}\cite{Gillman1960}
A space $X$ is said to be completely regular if it is Hausdorff and whenever $F$ is a closed set and $x \in F^c$, there exists a function $f \in C(X)$ such that $f(x) = 1$ and $f(F) = \{0\}$.
\end{defx}
\begin{defx}
The closed sets $f^{-1}(0) = \ker(f) = \{x \in X : f(x) = 0\}$ are called the zero sets of $C(X)$. On the other hand, the open sets $pos f = \{x : f(x) > 0\}$ and $neg f = \{x : f(x) < 0\} = pos (-f)$ are called cozero sets since they are complements of zero sets. Alternatively, every cozero set is of the form $X - Z(f) = pos |f|$.
\end{defx}
The definition shows how the regularity of the space $X$ is connected to the zero-sets of $C(X)$. Hence, the zero-sets are characterized by using subsets of the form $$f^{-1}(r) = \{x \in X : f(x) = r\}  \equiv  f^{-1}(0) = \ker(f) = \{x \in X : f(x) = 0\}.$$
The importance of these subsets lies in the fact that they are closed subsets. The zero-sets of $C(X)$ define the essential properties of the space $X$, which are those properties invariant under translation of zero-sets or transformation of $X$. We refer the reader to \cite{Gillman1960} for a complete understanding of how the zero-sets define the algebraic and topological structures of the compact regular space $X$. Two maps are defined from the characterization of zero sets of $C(X)$ in $X$ which contain the essentials for our goal in this paper.

The first map is the zero-set map $Z : C(X) \to Z(X)$ defined by $f \mapsto Z(f)$. It maps a function $f \in C(X)$ to its zero-set or kernel in $X$.
The image of $C(X)$ under the zero-set map $Z$ is the family $Z(X)$ of all zero-sets which is a base for closed sets. The weak topology on $C(X)$ relates to the topology defined on $X$ by $Z(X)$ as the base of closed sets, possessing the following properties. (a) $Z(f) = Z(|f|) = Z(f^n), \; \forall \; n \in \mathbb{N}$; (b) $Z(0) = X, Z(1) = \emptyset$; (c) $Z(fg) = Z(f) \cup Z(g)$; (d) $Z(f^2 + g^2) = Z(|f| + |g|) = Z(f) \cap Z(g)$.

The following also characterize zero sets: (1) Every zero set is a $G_\delta$ (a countable intersection of open sets) since $Z(f) = \underset{n \in \mathbb{N}}\bigcap \{x \in X : |f(x)| < \frac{1}{n} \}$. (2) When the space $X$ is normal every closed $G_\delta$ is a zero-set. (3) When $X$ is a metric space, every closed set is a zero-set consisting of all points whose distance from it is zero. (4) Every set of the form $\{x : f(x) \geq 0\}$ is a zero-set, since $\{x : f(x) \geq 0\} = Z(f\wedge 0) = Z(f-|f|)$ (supremum) and $\{x : f(x) \leq 0\} = Z(f\vee 0) = Z(f + |f|)$ (infimum).

The second map is defined based on the correspondence between maximal ideals $\mathfrak{m}_x = \{f \in C(X) : x \in \ker f\}$ of $C(X)$, which are kernels of the surjective ring homomorphism $x : C(X) \to \R$ defined by $x(f) = f(x)$, and the points of $X$. This is given in \cite{AtiyahDonald69} as the homeomorphism $\mathfrak{z} : X \to C(X), x \mapsto m_x$ of $X$ onto the closed points of $C(X)$. The homeomorphism identifies $X$ with the algebraic closed points of $C(X)$ under the Zariski topology. Because the topological and the algebraic structures of $C(X)$ are related to the convergence of $z$-filters in $z$-ultrafilters of $X$, we have the following result the proof of which is evident.
\begin{prop}
The $C(X)$-topology on $X$ coincides with the topology induced on $X$ by the Zariski topology on $C(X)$.
\end{prop}
Given the complete regularity of $X$ which ensures the convergence of $z$-filters in the space, the zero-sets $Z(X)$ of $X$ form a base for the closed sets. Since each $z$-ultrafilter $Z[\mathfrak{m}_x]$ converges to $x$ for each $x \in X$, we have the composition of the two bijections $Z\circ \mathfrak{z} : X \to X$ to be an identity on $X$.
\section{Action of the Borel Group}
Given that each maximal ideal $\mathfrak{m}_x$ is a $C(X)$-module, it has $G(1)$-action, which we will now describe. For every $g \in G(1)$, $Z(g) = \emptyset$. Thus, $G(1)$ defines some actions on the maximal ideal $\mathfrak{m}_x$ which reflect on the $z$-ultrafilter $Z[\mathfrak{m}_x]$, corresponding to the maximal ideal, and converging to $x \in X$. We define the following actions.

First, given $f \in \mathfrak{m}_x$ and $g \in G(1)$, $Z(f)\cap Z(g) = \emptyset$, then $|f| + |g|$ is a unit since it has no zero set; hence $Z(|f| + |g|) = \emptyset$. We define an action $G(1) \times \mathfrak{m}_x \to \mathfrak{m}_x$ by \[ (g,f) \mapsto h = \frac{|f|}{|f| + |g|}, \; \text{or } g\cdot f = h =  |f|\cdot(|f| + |g|)^{-1}. \] Then $h$ is defined by $h[Z(f)] = \{0\}$ and $h[X-Z(f)] = (0,1)\}$. This is a Borel action of $G(1)$ on the maximal ideal $\mathfrak{m}_x$ which preserves the zero-sets $Z(f) = Z(h)$ and cozero-sets $X - Z(f) = X - Z(h)$ of $f \in \mathfrak{m}_x$. Thus, it is a $G(1)$-action on $Z(X)$-base of the topology on $X$; and a $G(1)$-action on the open neighbourhood basis of the $C(X)$-topology on $X$.

Second, there is a multiplicative action $G(1) \times \mathfrak{m}_x \to \mathfrak{m}_x$ defined by $\tau_g : f \mapsto gf$ or $f \mapsto fg^{-1}$, which also preserves the zero-set $Z(f) \mapsto Z(gf)$ or $Z(f) \mapsto Z(fg^{-1})$ and cozero sets $pos|f| \mapsto pos|gf|$ (or $pos|fg^{-1}|$ resp.) of $f \in \mathfrak{m}_x$. The action scales the values of the functions in $\mathfrak{m}_x$ on their cozero sets.

The combination of these actions gives affine Borel $G(1)$-action on each maximal ideal $\mathfrak{m}_x$, which reflects on the zero-sets and cozero-sets associated to the maximal ideals. By these $G(1)$-actions on $\mathfrak{m}_x$, we can therefore consider $G(1)$ as defining transformations for elements of the base (open/closed sets) of the topology of $X$ which are measurable or Borel sets.

Subsequently, we consider the preservation of the Borel structure induced by the $C(X)$-topology on $X$. The $G(1)$-action preserve the $C(X)$-induced topological structure on $X$ by continuity. For by continuity the $G(1)$-actions preserve the local topological and Borel structure, $\mathfrak{A} =\{X - Z(f)\}_{f \in \mathfrak{m}_x} = \{U_f : f \in \mathfrak{m}_x, x \in X\}$, which forms a subalgebra of the $\sigma$-algebra $(X,\mathcal{B})$ induced on $X$ by $C(X)$.

Given the weak topology on $X$, the $(X, \mathcal{B})$ is the Borel space associated with the topology on $X$ if $\mathcal{B}$ is the smallest $\sigma$-algebra of subsets of $X$ with respect to which all real-valued continuous functions are measurable. $\mathcal{B}$ is generated by sets of the form $f^{-1}(C)$, where $C$ is any closed subset of the real line and $f \in C(X)$. Within the background of translations of the zero-sets by elements of $G(1)$, we can define a Borel transformation of these open/closed base as follows.
\begin{defx}\cite{Varadarajan62}
An automorphism of a Borel space $(X,\mathcal{B})$ is a one-to-one map $\phi$ of $X$ onto itself such that $\phi^{-1}(A) \in \mathcal{B}$ if and only if $A \in \mathcal{B}$.
\end{defx}
\begin{defx}
Given the Borel subspace $(Y,\mathfrak{A})$, where $Y = \underset{f \in \mathfrak{m}_x}\bigcup U_f$, a map  $T : Y \to Y$ such that for $T^{-1}(A) \in \mathfrak{A} \; \iff \; A \in \mathfrak{A}$ is an automorphism of the Borel space $(Y,\mathfrak{A})$.
\end{defx}
It follows therefore that the family $\mathfrak{A} = (\{U_f : f \in \mathfrak{m}_x, x \in X \}$ separates the points of $X$. For it implies that for any pair $x,y \in X$, there is $U_f \in \mathfrak{A}$ such that $x \in U_f$ and $y \not \in U_f$. $X$ is therefore separated. The automorphism defined above is within a maximal or an ultrafilter, so ensures the meeting of images and pre-images.
\begin{prop}
Given that $\mathcal{A}$ is the Borel structure generated by the family $\mathfrak{A} = \{U_f\}_{f \in \mathfrak{m}_x}$. The map $\varphi : g \mapsto T_g$ is a representation of $G(1)$ on the group of automorphisms $Aut(\mathfrak{A})$ of the the generating family of the Borel $\sigma$-algebra $\mathcal{A}$, whereby $T_g : U_f \to U_{gf}$.
\end{prop}
\begin{proof}
The families $\mathfrak{B} = (\{U_f : f \in \mathfrak{m}_x, x \in X\}$ constitute a basis for the topology on $X$. By definition, the open set $U_f$ is a cozero set of $f$, for each $f \in \mathfrak{m}_x \subseteq C(X)$. Thus, since $\mathcal{B}$ is the Borel structure associated with the $C(X)$-topology on $X$, and $G(1)$ is the group of units in $C(X)$, then the map, $\varphi : G(1) \times \mathcal{B} \to \mathcal{B}$ is a fibred action; when restricted to a generating subfamily $\mathfrak{A}$, corresponding to $Z[\mathfrak{m}_x], x \in X$, it is defined by $\varphi(g)(U_f) = U_{gf}$ and satisfies the properties of a group action:\\
(i) $\varphi(g_2g_1)(U_f) = \varphi(g_2)(U_{g_1f}) = U_{g_2g_1f} \subseteq U_{g_2f} \cap U_{g_1f} \in \mathfrak{A}$; \\ (ii) $\varphi(1)(U_f) = U_f$, for any $U_f \in \mathfrak{A}$.\\
Hence, the map $g \mapsto T_g$, where $g \in G(1), U_f \in \mathfrak{A}$ for each $x \in X$, defines a representation of $G(1)$ on the group of automorphisms $Aut(\mathfrak{A})$ of the generating set at each maximal ideal $\mathfrak{m}_x$ constituting the Borel $\sigma$-algebra $\mathcal{B}$.
\end{proof}
\begin{coro}
The families of open sets $\{ U_f : f \in \mathfrak{m}_x, x \in X\}$ with the action of $G(1)$ generate the Borel structure $(X, \mathcal{B})$, the $\sigma$-algebra of $X$.
\end{coro}
A family $\mathfrak{B} = \{U_f : f \in \mathfrak{m}_x, x \in X\}$ of subsets of a set $X$, according to Mackey \cite{Mackey57}, is said to generate a unique Borel structure $\mathcal{B}$ on $X$ if it is the smallest Borel structure $\mathcal{B}$ for $X$ which contains $\mathfrak{B}$. In this case, since $\mathfrak{B}$ is the base of the topology on $X$, the Borel structure $\mathcal{B}$ is associated with the topology on $X$.

\subsection{$C(X)$-Embedding and Measurability}
Given the weak topology on $X$ and the Borel structure generated by the base of closed (or open) sets associated with this topology, we define a homeomorphism as follows. For each $f \in \mathfrak{m}_x$ and $g \in G(1)$, $Z(f)\cap Z(g) = \emptyset$, hence $Z(|f| + |g|) = \emptyset$, define \[ \tau_g : f \mapsto \tau(f,g) = \frac{|g|}{|f| + |g|}. \] Then $\tau_g$ is a measurable function define by each $g \in G(1)$ on the fibre (or on the local bisection) of the groupoid, defined by $\tau_g(f) = \tau(f,g)$. This functions map the zero sets $Z(f)$ to $\{1\}$ and the cozero sets $X-Z(f)$ to the open set $(0,1)$. Thus, it is a homeomorphism on the regular space $X$, such that $\tau(f,g) : X \to (0, 1] \simeq \mathbb{T}$. Thus, the compact space $X$, with the weak topology, is isomorphic to the torus $\mathbb{T}$.

Corresponding to this homeomorphism is a $G(1)$-action defined on the neighbourhood basis (or on the base of closed sets $Z(X)$) of the weak topology on $X$. By definition, the zero sets are balanced sets of the topological space ($X,C(X)$-topology), which is a topological vector space(tvs). The above homeomorphism, taking open subsets of $X$ to open subsets of $\mathbb{T}$, helps us to understand the Borel structure of $G(1)$ and the notion of $C(X)$-embedding of closed sets of $\R$ which is defined in \cite{Gillman1960} as follows.
\begin{defx}
A subspace $S \subset X$ is said to be $C(X)$-embedded if every function in $C(S)$ can be extended to a function in $C(X)$.
\end{defx}
Since every closed set in a metric space is a zero-set, and disjoint sets are completely separated; if the subset $S$ is closed, then its closed subsets are closed in $X$, and completely separated sets in $S$ have disjoint closures in $X$; it follows by Tietse's Extension Theorem that every closed set in $S$ is $C(X)$-embedded. This rests on the fact that every closed set in $\R$ is $C(X)$-embedded. The following theorem expresses this idea.
\begin{thm}\cite{Gillman1960}
If there exists a homeomorphism (function) in $C(X)$ taking $S$ onto a closed set in $\R$, then $S$ is $C(X)$-embedded in $X$.
\end{thm}
\begin{proof}
Let $\tau$ be the homeomorphism (function) in $C(X)$. Then $\theta = \tau^{-1}|_{\tau(S)}$ is a continuous mapping from $H = \tau(S)$ onto $S$, with $\theta(\tau(s)) = s \in S$. Let $f \in C(S)$ be arbitrary. The composite function $f \circ \theta$ is in $C(X)$. Since $H \subset \R$ is closed, by hypothesis, it is embedded; so, there exists $g \in C(\R)$ that agrees with $f\circ \theta$ on $H$. Then $g \circ \tau \in C(X)$, and for all $s \in S$, we have $(g\circ \tau)(s) = f(\theta(\tau(s))) = f(s)$, which means that $g\circ \tau$ is an extension of $f$.
\end{proof}
The homeomorphism $\tau$ is a unit in $C(X)$. Thus, the $G(1)$-action is related to the $C(X)$-embedding of closed sets of $\R$, and expresses its Borel property. It also highlights the local convexity of the regular compact space $X$ as a $G(1)$-space. We give this as a corollary.
\begin{coro}
The $G(1)$-action on every completely regular compact space $X$, $C(X)$-embeds closed subsets of $\R$ in $X$ thereby making it locally convex.
\end{coro}
\begin{defx}
A topological space $X$ is said to be pseudocompact if its image under any continuous real-valued function is bounded.
\end{defx}
The compactness of $X$ guarantees that $\tau$ is a homeomorphism $X \to [0,1] \simeq \mathbb{T}$ which $C(X)$-embeds closed subsets $[a,b]$ of $\R$ in $X$. A space $X$ is pseudocompact if and only if it contains no $C(X)$-embedded copy of $\mathbb{N}$ as shown in \cite{Gillman1960}. Thus, the general Borel transformations can be modelled on the homeomorphism $\tau$ approximated by the action $(0,1] \times X \to X$, where $((t,x),y)$ is the graph of a continuous transformation $g(t) : x \mapsto y$.
\begin{rem}
We therefore make the following remarks. \\
(1) The action is a presentation of $X \times [0,1]$ as a locally convex topological vector space by the continuous functionals $\tau : C(X) \to [0,1]$ separating its points. \\ (2) With the notion of $C(X)$-embedding, we see that $\tau$ defines a norm, and hence, a metric on $X$, which brings us to measures on $X$, and the decomposition of a finite (non-ergodic) Borel measure $\mu$ on $X$ with respect to an ultrafilter $Z[\mathfrak{m}_x]$, and the generating set $\mathfrak{B} = \{U_f : f \in \mathfrak{m}_x, x \in X\}$ of the $\sigma$-algebra $\mathcal{B}$.
\end{rem}

\section{The Generalized Space and Dynamics}
The above action shows that $G(1)$ is a Borel group. The dynamical system defined by the Borel group $G(1)$ is by its ergodic action generalized in the action of the algebra/lattice $C(X)$) on the generalized space $\mathcal{M}(X)$, the space of nonnegative Radon measures on $X$. According to \cite{EinsWard2011}, the time evolution of dynamical systems modelled by measure-preserving actions of integers $\mathbb{Z}$ or real numbers $\R$ which represent passage of time are generalized by measure-preserving actions of \emph{lattices} which are usually "subgroup" of Lie groups.

Each maximal ideal $\mathfrak{m}_x$ characterizes and encodes the symmetries of the measurable functions vanishing at each point of $X$ and on its respective neighbourhoods. The symmetry is represented by the $z$-ultrafilter of zero sets (algebraic varieties of $C(X)$) converging to each $x \in X$. The complements of these algebraic sets constitute the open neighbourhood base of points of $X$. The ultrafilter $\mathcal{F}$ convergence to $x$ has associated nets of measurable functions converging to a fixed point function $f$ defined on $x$. This follows from the correspondence between filters and their derived nets. Given a net $f_\alpha$ of contractions in the complete metric space $X$, as described in \cite{HasselblattKatok}, it follows that $f_\alpha \to f$ such that $f(x) = x$. We will represent all these on the generalize space $\mathcal{M}(X)$ with ergodic action of $C(X)$. But we need to recall the ideas of a generalized subset and an ergodic subgroup.

Mackey's conception of a measure class $C$ as a \emph{generalized subset} was extended to the whole space $\mathcal{M}(X)$ of Radon measures on $X$ which is conceived as the \emph{generalized space} of points or \emph{state space} (see \cite{EinsWard2011}). At the centre of this extension is the focus on (i) measure preserving transformations of the compact metric space $X$, and (ii) the Dirac measures $\delta_x$ as generalized or geometric points isomorphic to points of $X$. The role assigned to the ergodic transformations by Mackey, is to translate along time in such a way as to ensure the invariance of measure or state.

Recall the surjective homomorphism $T_x : C(X) \to [0,1)$ which has the maximal ideal $\mathfrak{m}_x$ as its kernel. It follows that a continuous linear transformation $\varphi$ on $X$, by its relation to $C(X)$ as shown above, induces another continuous transformation on the generalized space $\varphi_* : \mathcal{M}(X) \to \mathcal{M}(X)$, and thereby defines an action of the algebra $C(X)$ on the $\mathcal{M}(X) \simeq L^\infty(X,\mu)$. Notice that the replacement of $\R$ with $[0,1)$ makes all transformations of $X$ contractions.  This translates the whole structure associated to $x \in X$, (the zero sets, the null sets which coincide with the zero sets, and the maximal ideal) from one point to the other.

Thus, the restriction of $\varphi_*$ to the subset $\mathcal{U} = \{\delta_x : x \in X\}$ gives a transformation of $\mathcal{U}$ defined as $\varphi_* : \delta_x \mapsto \delta_{\varphi(x)}$, such that for any $A \subseteq X$, we have \[ (\varphi_*\delta_x)(A) = \delta_x(\varphi^{-1}A) = \delta_{\varphi(x)}(A). \] Without the restriction the subset $\mathcal{U} = \{\delta_x : x \in X\} \subset \mathcal{M}(X)$ of the generalized points, can be continuously and affinely extended to the generalized space $\mathcal{M}(X)$. Subsequently, if the generalized points $\mathcal{U} = \{\delta_x : x \in X\}$ can generate the generalized space $\mathcal{M}(X)$, then the generalized subsets $\{ C_f : f \in \mathfrak{m}_x\}$ (the measure classes) can be generated from the generalized points, the Dirac measures $\delta_x$. We give this as a result.
\begin{prop}
The generalized space $\mathcal{M}(X)$ of Radon measures on $X$ is an affine and continuous extension of the geometric points $\mathcal{U} := \{\delta_x : x \in X\}$.
\end{prop}
\begin{proof}
The coincidence of zero sets of $C(X)$ with null sets of $\mathcal{M}(X)$ establishes the existence of: (a) measure classes and (b) nets $\{\mu_\alpha\}$ of non-ergodic Radon measures related to $\{\mu_f : f \in \mathfrak{m}_x, x \in X\}$ which converge to the Dirac measures $\{\delta_x : x \in X\}$ as the $z$-ultrafilter $Z[\mathfrak{m}_x]$ converges $x$. Since the elements of $\mathfrak{m}_x$ vanish at $x$, its $G(1)$-action is transferred to the fibres of measure classes via $\varphi_*$. The fibres of measures classes constitute tangent measures to $\delta_x$, with a representation of $G(1)$, similar to slice representation in group action.

The transformations $\varphi_* : \mathcal{M}(X) \to \mathcal{M}(X)$ are the natural continuous and affine extensions of the generalized points $\mathcal{U} = \{\delta_x : x \in X\}$ to the generalized space $\mathcal{M}(X)$. This is shown by considering the generalized space $\mathcal{M}(X)$ as generated by the decomposition action of the $C(X)$-modules $\mathfrak{m}_x$, whereby $\mu \mapsto \mu_{gf}$ preserves the null set $Z(f)$. This decomposition preserves the two actions of the Borel group $G(1)$ defined above on $\mathfrak{m}_x$, which generate $\mathcal{M}(X)$ on the base space $\mathcal{U} = \{\delta_x : x \in X\}$. So the base space $\mathcal{U}$ is invariant under direct $G(1)$-action corresponding to the $G(1)$ defined Borel transformations on $X$. Thus, the Dirac measures $\delta_x \in \mathcal{U}$ constitute the geometric/generalized points of the generalized space $\mathcal{M}(X)$.
\end{proof}
This result leads to complementary ways of understanding the concept of ergodicity which is characterized in different ways for a measure and for a transformation in \cite{EinsWard2011}. For the measure space $(X,\mathcal{B},\mu)$ with a transformation $\varphi : X \to X$, it means indecomposability of $X$ into two $\varphi$-invariant measurable subsets. But for a measure preserving transformation $\varphi : X \to X$ of a probability space $(X,\mathcal{B},\mu)$, it said to be ergodic if for any $B \in \mathcal{B}, \varphi^{-1}B = B \; \implies \; \mu(B) = 0$, or $\mu(B) = 1$. The $\varphi$-invariant measure $\mu$ is also called ergodic. See \cite{EinsWard2011}.

Hence, a non-ergodic measure is infinitely decomposable with respect to a measure-preserving transformation $\varphi$, such that a set of $\varphi$-invariant measures contains the ergodic measures which is $\varphi$-indecomposable as its boundary points. Thus, a closed subset of $\varphi$-invariant measures in $\mathcal{M}(X)$ is either a singleton (when the measure is ergodic) or an infinite set when it contains non-ergodic measures. In the latter, the boundary points are ergodic. Each boundary point makes a complete ergodic meaning with the nets converging to it. This net completeness defines the nature of the (geometric) point and the dynamics associated to it.

The dynamism defined by the transformation $\varphi$ is encoded in the symmetry of the measures classes contributing to the convergent net. So, the connection between ergodic theory and the dynamical system defined by continuous transformations on compact metric spaces, according to \cite{EinsWard2011}, is captured by the closure of the resulting convex set of non-ergodic $\varphi$-invariant measures with ergodic measures as a boundary point. This means a closed set assures the convergence of every net in the set. The connection is therefore based on the ergodicity of both measures and transformations. The following is an important result of this section.
\begin{thm}
The action $\varphi$ of the commutative algebra $C(X)$ on the generalized space $\mathcal{M}(X)$, defined by the Borel group $G(1)$ and the maximal ideals $\mathfrak{m}_x$, has an ergodic limit.
\end{thm}
\begin{proof}
According to \cite{EinsWard2011}, ergodic theorems express a relationship between averages taken along the orbit of a point under the iteration of a measure-preserving map/tansformation $\varphi : X \to X$. The iteration of the transformation on the compact space $X$ isomorphic to the generalized points $\mathcal{U}$ represents passage of time, and its invariance $\varphi : x \mapsto \varphi(x) \simeq \varphi_*(\delta_x) = \delta_{\varphi(x)}$ constitutes a net of transformations which represents an average over time. Restricted to the generalized points $\mathcal{U}$, the dirac measures $\delta_x \mapsto \delta_{\varphi(x)}$ are $\varphi$-invariant measures representing average over states as limits of states.

This invariance accommodates the induced iteration on the generalized space $\mathcal{M}(X)$ tangent at each $\delta_x$, with respect to some invariant measures $\mu$ (or measure classes $\varphi_* : C_f \to C_{f\circ \varphi}$) representing the states. This vertical component of ergodicity represents average over space (averages taken over the classes of measures) which is the net of $\varphi$-invariant non-ergodic measures converging to an ergodic limit guided by the convergence of $z$-ultrafilter $Z(\mathfrak{m}_x)$. We have shown that these two averages are defined by invariance of $\mathcal{U}$ and the measure classes $\{C_f\}$ defined by $\mathfrak{m}_x$ under $G(1)$-actions .
\end{proof}
We will now consider how the commutative algebra (lattice) $C(X)$ action via the maximal ideals $\mathfrak{m}_x$-the $G(1)$-spaces-define the tangent measures to the generalized points $\mathcal{U}$. We recall here that $C(X)$ can be said to act at each point $x \in X$ through the action groupoid $\mathcal{G}:= G(1)\times \mathfrak{m}_x \to \mathfrak{m}_x$.

\section{Tangent Measures and Representation}
We follow the definition of tangent measures given by \cite{O'Neil95} and \cite{Sahlsten2014}. These defined the tangent measures of a Radon measure $\mu$ on $\R^n$. We adapt the definition and result to apply to a complete metric (measure) space $(X,\mu)$. To define the tangent space of a Radon measure $\mu$ on a complete metric space $X$ with metric $d$, we let the sets \[B(x,r) = \{y \in X : d(x,y) \leq r \}; \; \text{and } U(x,r) = \{y \in X : d(x,y) < r\}\] be the closed and open balls respectively, centred at $x \in X$ with radius $r >0$. Subsequently, the support of a measure $\mu \in \mathcal{M}(X)$ is given as $\{x \in X : \mu(B(x,r)) > 0$ for any $r > 0 \}$. With these, a homothetic transformation is defined as follows.
\begin{defx}
Given a point $x \in X$, the mapping $T_{x,r} : X \to X$ defined by $y \mapsto \frac{(y-x)}{r}$ is an affine (homothety) transformation from $B(x,r)$ to $B(x,1)$.
\end{defx}
It follows that $T_{x,r} \in Aut(X,\mu)$. Since any $r \in \R$ could be considered the image $f(x)$ of some continuous function $f \in C(X)$, we let $T_{x,r} = T_{x,f}$. The push-forward of the Radon measure $\mu$ under the map $T_{x,r}$ is given as ${T_{x,r}}_*\mu(A) = \mu(rA + x), A \subset X.$

Also, given a scalar $c>0$, we have $c{T_{x,r}}_*\mu = \top_{x,r,c}(\mu)$. This induces a map \[ c{T_{x,r}}_* = \top_{x,r,c} : \mathcal{M}(X) \to \mathcal{M}(X). \] We can also have $c = g(x)$ for some $g \in G(1)$. When $r = 1$, this gives $T_{x,1}(y) = y-x$. This is used to define tangent measures of a Radon measure as follows.
\begin{defx}
A measure $\nu \in \mathcal{M}(X)\setminus \{0\}$ is a tangent measure to $\mu \in \mathcal{M}(X)$ at $x \in X$ if there exists a net $r_k \searrow 0$ and $c_k > 0$ such that \[ c_k{T_{x,r_k}}_*\mu = \top_{x,r_k,c_k}(\mu) \to \nu, \text{as } k \to \infty. \]
Alternatively, assume $\mu \in \mathcal{M}(X), x \in X$, and $r>0$; define for $A \subset X$
\[\mu_{x,r}(A) := \mu(x + rA) := \mu(\{x + ra : a \in A\}). \] Then, a measure $\nu \in \mathcal{M}(X)$ is a tangent measure to $\mu$ at $x$ if $\nu$ is nonzero and there exist $r_k \searrow 0$ and $c_k > 0$ such that $\underset{k \to \infty}\lim c_k\mu_{x,r_k} = \nu.$
\end{defx}
The set of all tangent measures to $\mu$ at $x$ is denoted by $Tan(\mu,x)$. It is a closed subset of $\mathcal{M}(X) \setminus \{0\}$ with the following properties.\\
(1) $c\nu \in Tan(\mu, x)$ whenever $\nu \in Tan(\mu,x)$ and $c>0$. \\ (2) $\nu_{x,r} \in Tan(\mu,x)$ whenever $\nu \in Tan(\mu,x)$ and $r > 0$ ($G(1)$-invariance). \\ (3) $Tan(\mu,x)$ is a closed set with respect to the space $\mathcal{M}(X)$ of all nonzero, Borel, regular, locally finite measures on $X$.

The definition of tangent measures using the induced map ${T_{x,r}}_* : \mathcal{M}(X) \to \mathcal{M}(X)$ can be related to the $G(1)$-action and $\mathfrak{m}_x$-decomposition of the generalized space $\mathcal{M}(X)$. Considering these two as the action of the algebra $C(X)$ on $\mathcal{M}(X)$, we have the following major theorem.
\begin{thm}
The space of tangent measures $Tan(\delta_x,x)$ of the Dirac measure $\delta_x$ at each $x \in X$ constitutes the orbits of the affine action of the commutative algebra/lattice $C(X)$ on the generalized space $\mathcal{M}(X)$.
\end{thm}
\begin{proof}
Assume $f(x) = r > 0$, for some $f \in C(X)$, then the transformation $T_{x,r} = T_{x,f} : X \to X, \; y \mapsto \frac{y-x}{f(x)}$, takes $B(x,f(x))$ to $B(x,1)$. Recall the $G(1)$-action defined on $\mathfrak{m}_x$ by \[(g,f) \mapsto h = \frac{|f|}{|f| + |g|} : X \to [0,1) \] which preserves the zero set $Z(f)$ of each $f \in \mathfrak{m}_x$, and hence its measure class $[\mu_f]$.

The existence of a $z$-ultrafilter for each $\mathfrak{m}_x$ that converges to $x \in X$ implies the existence of a net $(f_k) \subset \mathfrak{m}_x$ such that $Z(f_k) \to x$ in $Z[\mathfrak{m}_x]$. Thus, we define the functions $h_k = \frac{|f_k|}{|f_k|+|g|}$ in such a way that $h_k(y) = r_k \searrow 0$ as $Z(f_k) \to x$ and $k \to \infty$. Then the net $r_k \searrow 0$ is given by $h_k : X \to [0,1)$. Since $X$ is completely regular, it is countable and the images $h_k(y) = r_k$ of the map on $X$ has $0$ as the infimum; such that $r_k \searrow 0$ as $k \to \infty$.

Similarly, the positive number $c > 0$ is taken to be an image $g(y)$ of some $g \in G(1)$. It then follows that $c_k$ result from $G(1)$-action on $\mathfrak{m}_x$ and subsequently on $\mathcal{M}(X)$ as \[g \circ {T_{x,f}}_* = T_{x,fg} : \mathcal{M}(X) \to \mathcal{M}(X).\] Thus, the positive numbers $c_k = |g|_k(y)$ are defined for $y \in Z[\mathfrak{m}_x]$.

By the correspondence between the zero sets $Z(f)$ and the null sets $N(\mu_f)$, the convergence $\underset{k \to \infty}\lim {g_k}\mu_{x,f_k} = \underset{k \to \infty}\lim \mu_{x,f_kg_k} = \nu$ holds across the measure classes, which corresponds to the convergence of a $z$-ultrafilter.

That the limit $\nu$ of the net of measures $\underset{k \to \infty}\lim \mu_{x,f_kg_k}$ is a Dirac measure $\delta_x$ at $x$ follows from the fact that the limit or the boundary point of a sequence of invariant measures in $\mathcal{M}(X)$ is invariant and ergodic (\cite{EinsWard2011}, Theorem 4.4). This again follows from $\mathcal{F} \to x$. So, each measure class has the action of $G(1)$, \[T_g : [\mu_f] \to [\mu_f] \; \text{defined by } \mu_f \mapsto \mu_{fg},\] where $f, fg \in \mathfrak{m}_x$ and $\mu_f \sim \mu_{fg}$; So $G(1) \times [\mu_f] \to [\mu_f]$ is a $G(1)$-action. From these, we conclude that the space of tangent measures $Tan(\mu,x)$ to a Radon measure $\mu \in \mathcal{M}(X)$ comprises of measures of all the measure classes, with each measure class $[\mu_f]$ compact and convex. Thus, $Tan(\mu,x) \simeq \mathcal{M}(X)\setminus \{0\}$.
\end{proof}
The convergence of nets of invariant measures in the tangent space $Tan(\mu,x)$ points to the relationship between the dynamical system defined by the $C(X)$-action and the definition of the tangent space $Tan(\mu,x)$, where the invariance condition is in view of ergodicity of the action as we have seen above. Since the action is defined by the automorphism group $Aut(X,\mu)$, if we let $\mathcal{E} \subset Aut(X,\mu)$ be the ergodic subgroup, it is shown to be a dense $G_\delta$ in $Aut(X,\mu)$.
\begin{thm}\cite{Kechris2010}
The set of ergodic transformations $\mathcal{E}$ is dense $G_\delta$ (intersection of open sets) in $Aut(X,\mu)$.
\end{thm}
Subsequently, to use the dynamical system determined by the measurable/continuous measure-preserving transformations $G(1) \times \mathfrak{m}_x \to \mathfrak{m}_x$ to study the ergodic action of the algebra $C(X)$, we have to conceive it as an action groupoid acting on each $x \in X$. Then the open and dense $t$-fibres will represent the ergodic subgroup of $Aut(X,\mu)$ in the commutative algebra $C(X)$ and on the generalized space of Radon measures $\mathcal{M}(X)$ as virtual groups.

A virtual group arises when a $G$-space has only invariant measure class, which is a general case of invariant measure. Assuming $G$ is a Borel group, and $C$ an invariant measure class in the standard $G$-space $S$ with non-transitive $G$-action; then $B$ and $S-B$ are invariant subsets of $S$. Thus, $S = B\oplus (S-B)$ is the direct sum of $G$-spaces. If $B$ and $S-B$ are not of measure zero, we obtain a non trivial invariant measure class in $B$ and $S-B$ by taking any $\mu \in C$ and then taking the class of its restriction to $B$ and $S-B$ respectively. This will realize $(S,C)$ as the direct sum of two invariant subsystems in such a way that: \\ (1) Every measurable invariant subset of $S$ is either of measure zero or the complement of a set of measure zero. \\ (2) Every invariant subset of $S$ on which $G$ acts transitively is of measure zero
\begin{defx}
An action of $G$ on $S$ is ergodic (or metrically transitive) with respect to a measure class $C$ if (1) holds; namely, every measurable invariant subset of $S$ is either of measure zero or the complement of a set of measure zero. When (1) and (2) hold, the action is \emph{strictly ergodic}.
\end{defx}
Since the $G(1)$-action on the maximal ideals $\mathfrak{m}_x$ at each point of $X$ preserves the zero sets, it is ergodic. Likewise the induced action on $\mathcal{M}(X)$ whose orbits are measure classes; which means that the $G(1)$-action on the $\mathfrak{m}_x$-decomposed $\mathcal{M}(X)$ is also ergodic. We now consider the virtual (sub)group arising from the ergodic action of the algebra $C(X)$ on the generalized space $\mathcal{M}(X)$, namely, the measure groupoid.

\section{The Measure Groupoid and its Action}
The use of zero sets and null sets in the characterization of action of the commutative algebra $C(X)$ evidently agrees with Mackey's definition of ergodic action, and his understanding of a measure class $C$ or $[\mu]$ of a typical Radon measure $\mu$ as a generalized subset agrees with our presentation of a measure class as an orbit of $G(1)$-action. According to this understanding, the measure class $[\mu]$ of $\mu$ is the set of all measures in $\mathcal{M}(X)$ possessing the same null sets in common with $\mu$. Each member is related to every other measure in the class for the class (and each member) is defined by the complement of its measurable set. For a countable set, its complement that is of $\mu$-measure zero, is the largest null set which completely determines the measure class $[\mu]$. Cf. \cite{Mackey57}).

The Dirac measures define multiplicative functionals on $C(X)$ by the isomorphism $X \cong \mathcal{U}$. The kernel of the functional is the maximal ideal $\mathfrak{m}_x$. That is, $\ker \delta_x = \{f \in C(X) : f(x) = 0 \} = \mathfrak{m}_x $. Therefore, using the isomorphism, we see that the set $\mathscr{C}$ of measure classes arising from the $\mathfrak{m}_x$-decomposition of a (non-ergodic) Radon measure $\mu$ on $X$ given as \[f : \mathcal{M}(X) \to \mathscr{C}, \mu \mapsto \mu_f \; \text{for each } f \in \mathfrak{m}_x;\] is trivialized on the generalized points $U = \{\delta_x : x \in X\}$. The convergence of a $z$-ultrafilter $Z(\mathfrak{m}_x)$ associated to $\mathfrak{m}_x$ to $x \in X$ gives the tangent space of measures $T_{\delta_x}\mathcal{M}(X)$ for each point of the generalized set $\mathcal{U} = \{\delta_x : x \in X \}$.

Subsequent on this convergence of a $z$-ultrafilter $Z[\mathfrak{m}_x]$, which assures the closure of a convex set $\mathcal{M}(X)_\varphi$ of $\varphi$-invariant measures in $\mathcal{M}(X)$ having $\delta_x \in \mathcal{U}$ as extreme/limit points, we define a $\varphi$-dynamism. Thus, a continuous transformation $\varphi$ on $X$, by its action on the Radon measures $\mu$ on $X$, induces another continuous transformation on the generalized space $\varphi_* : \mathcal{M}(X) \to \mathcal{M}(X)$, given as $\varphi_*(\mu)(A) = \mu(\varphi^{-1}(A))$ for any Borel subset $A \subseteq X$. This induced map translate the tangent fibre/space. We will use this fact to prove the following theorem.
\begin{thm}\cite{EinsWard2011}
Let $X$ be a compact metric space and $\varphi : X \to X$ a continuous map. Then for any $\mu \in \mathcal{M}(X)_\varphi$ there is a unique probability measure $\delta$ defined on the Borel subsets of the compact metric space $\mathcal{M}(X)_\varphi$ with the properties that \\
(1) $\delta(\mathcal{E}(X)_\varphi) = 1$, where $\mathcal{E}(X)_\varphi = \mathcal{U}$ are extreme points of $\mathcal{M}(X)_\varphi$ \\
(2) $\displaystyle \int_X fd\mu = \int_{\mathcal{U}} \left(\int_X fd\nu \right)d\delta(\nu)$
\end{thm}
\begin{proof}
We have seen that every continuous transformation $\varphi : X \to X$ which is ergodic must be a homothety of the form $T_{x,f(x)} : X \to X, y \mapsto \frac{y-x}{f(x)}$, by the characterizations of ergodic transformations of the compact metric space $X$.

Now, to say that the compact convex set of measures $\mathcal{M}(X)_\varphi$ is $\varphi$-invariant is to say that the class of a Radon measure $\mu \in \mathcal{M}(X)_\varphi$ is preserved by the transformation $\varphi$; that is, $\varphi_*(\mu) \sim \mu$. Hence, the transformation preserves the measure classes; $\varphi_* : [\mu_f] \to [\mu_f]$. Relating this to the tangent space of measures $Tan(\delta_x,x)$, we see that the transformation $\varphi : X \to X$ with the induced map \[\varphi_* : Tan(\delta_x,x) \to Tan(\delta_{\varphi(x)},\varphi(x)) \] defines a left translation on the tangent bundle over $\mathcal{U}$ of measures on $X$. Hence the following diagram commutes
\begin{center}
\begin{tikzpicture}
\draw[->] (0.3,0) -- (4.5,0) node[right] {$\mathcal{U}$};
\draw[-] (2.5,0) node[above] {$\varphi$};
\draw[<-] (5,0.3) -- (5,1.5) node[above] {$Tan(\delta_{\varphi(x)},\varphi(x))$};
\draw[-] (5,0.9) node[right] {$t$};
\draw[-] (0,1.5) node[above] {$Tan(\delta_x,x)$};
\draw[->] (0,1.5) -- (0,0.3) node[below] {$\mathcal{U}$};
\draw[-] (0,0.9) node[left] {$t$};
\draw[->] (1,1.7) -- (3.4,1.7);
\draw[-] (1.7,1.7) node[above] {$\varphi_*$};
\end{tikzpicture}
\end{center}
This is a tangent bundle structure over $\mathcal{U}$, where the tangent space of measures $Tan(\delta_x,x)$ varies continuously on $\mathcal{U}$ with the continuous transformation $\varphi$. The first (inner) integration is over the invariant measure classes (orbits) making up the fibre, and the second (outer) is over the ergodic measures (base space).
\end{proof}
From the foregoing, and our definition of the homothetic transformations $\varphi = T_{x,r}$ using the $G(1)$ and $\mathfrak{m}_x$, it follows that the compact convex set $\mathcal{M}(X)_\varphi$ is composed of measures translated in a parallel or transverse manner, which preserve their equivalence classes, across the tangent spaces. They are therefore conjugate measures since $\mu \in \mathcal{M}(X)_\varphi \; \implies \; \varphi_*([\mu]) = [\mu \circ \varphi] = [\mu]$. Hence, \[\varphi_*(\mu_f)(f) = \int_Xf d(\mu_f \circ \varphi)  = \int_X fd\mu_f.\]

This represents the interesting and difficult aspect of this dynamical system; which, according to \cite{EinsWard2011}, is to understand the orbit of a point $x \in X$ under the iteration of the ergodic and continuous transformations $\varphi$. The orbit is a \emph{section} of the equivalence classes of measures constituting the tangent bundle $T(\mathcal{M}(X))$. We will now consider the virtual group as a measure groupoid to throw light on the tangent bundle structure as described above.

Our characterization of the sub $\sigma$-algebra generated by the complements of the zero sets $Z(\mathfrak{m}_x)$ above coincides with Peter Hahn's description in \cite{Hahn78}.  Thus, using \cite{Hahn78}, Theorem 2.1,  we summarize the conclusion above as follows. \\ (i) every measure class $[\mu_f]$ in $\mathcal{M}(X)$ contains a probability $\lambda : \lambda(X) = 1$;\\ (ii) the map $p : U_f \mapsto f$ is a Borel surjection;\\ (iii) if $\mu \sim \nu$ on $X$, then $\bar{\mu} \sim \bar{\nu}$ are their push-forwards on $\mathfrak{m}_x$, where $\bar{\mu} = p_*\mu, \bar{\nu} = p_*\nu$;\\ (iv) their R-N derivative $j = \frac{d\nu}{d\mu}$ is a positive Borel function on $Y \subset X$;\\ (v) the $\mathfrak{m}_x$-decomposition $\nu \mapsto \nu_f$ of a Radon measure on $X$, gives a surjective map onto the equivalence classes (orbits of $G(1)$-action) $\pi : \mathcal{M}(X) \to \mathscr{C}, \nu \mapsto \nu_f$ such that the following are satisfied: \\
(1) If $h \geq 0$ is  Borel on $U_f$, then $\displaystyle f \mapsto \int hd\nu_f$ is its extension as a real-valued Borel function on $X$ (implying the closure of $U_f$ is $C(X)$-embedded). \\
(2) $\nu_f(Z(f)) = \nu_f(X-p^{-1}(\{f\})) = 0, \; \forall \; f \in \mathfrak{m}_x$ ($\implies \; Z(f) = \nu_f$-null set).\\
(3) If $h \geq 0$ is Borel on $X$, then $\displaystyle \int hd\nu = \int \left(\int hd\nu_f\right)d\bar{\mu}(f)$.
\begin{rem}
The conclusion (iii) above implies that the R-N derivative is a decomposition invariant property of equivalent measures; thus $j = \frac{d\nu}{d\mu} = \frac{d\nu_f}{d\mu_f} = \frac{d\nu_{fg}}{d\mu_{fg}}$ a.e. If we define this derivative as operator $\partial_\mu : \mathcal{M}(X) \to L^1(X,\mu)$, then we see that the operator commutes with the $C(X)$-actions. For given $f \in \mathfrak{m}_x, g \in G(1)$, we have \[(g \circ f)\partial_\mu(\nu) = g * \partial_{\mu_f}(\nu_f) = \partial_{\mu_{fg}}(\nu_{fg}).\]
This commutation clearly expresses the dynamical symmetry of the system which we will also represent using the groupoid action. We have the following definitions from \cite{Hahn78}.
\end{rem}
\begin{defx}
Given a measure $\mu$ on a Borel space or set $A$, the inverse of the measure $\mu^{-1}$ is defined on the set as $\mu^{-1}(A) = \mu(A^{-1})$.
\end{defx}
\begin{defx}
A measure $\mu$ on a groupoid $\mathcal{G}$ is symmetric if $\mu = \mu^{-1}$. A class of measure $\mathcal{C}$ is symmetric if it contains a symmetric measure. Thus, if $\mu \in \mathcal{C}$, then $\mu = \mu^{-1}$ implies $\mathcal{C}$ is symmetric.
\end{defx}
\begin{defx}
Let $\mathcal{C}$ be a symmetric measure class on an analytic groupoid $\mathcal{G}$. Let $\lambda \in \mathcal{C}$ be a probability measure with $t$-decomposition \[ \lambda = \int \lambda^ud\bar{\lambda}(u) \; \text{over } \mathcal{G}^o. \] Then $\lambda$ is (left) \emph{quasi-invariant} if there is a $\bar{\lambda}$-conull Borel set $U \subset \mathcal{G}^o$ such that if $t(g) \in U$ and $s(g) \in U$, then $g \cdot \lambda^{s(g)} \sim \lambda^{t(g)}$.
\end{defx}
Recall the presentation of the product space $(G(1) \times \mathfrak{m}_x, \nu \times \mu_f)$ as a standard Borel space. The family of measures $(\nu \times \mu_f)_{f \in \mathfrak{m}_x}$ results from the decomposition of a Radon measure $\mu$ on $X$ by the actions of the algebra/lattice $C(X)$, which we have described above. Thus, we have \[ (G(1) \times \mathfrak{m}_x) : \mu \to (\nu \times \mu_f)_{f \in \mathfrak{m}_x}; \] where $[\mu_f]_{f\in \mathfrak{m}_x} = \mathscr{C}$ is the collection of the measure classes of $\mathcal{M}(X)$, and $\nu$ is ergodic and a Haar measure on $G(1)$. We give another main theorem of the work.
\begin{thm}
The product space $(G(1) \times \mathfrak{m}_x, \nu \times \mu_f)$ is a standard Borel space and a measure groupoid $\mathcal{G} = \mathfrak{m}_x \rtimes G(1)$. The Haar system of measures is the family $(\nu \times \mu_f)_{f \in \mathfrak{m}_x}$ resulting from the $\mathfrak{m}_x$-decomposition of a Radon measure $\mu$ on $X$ and $G(1)$-actions at each $x \in X$.
\end{thm}
\begin{proof}
We define the measure groupoid $\mathcal{G} = \mathfrak{m}_x \rtimes G(1) \rightrightarrows \mathfrak{m}_x$, where the arrows \[\mathcal{G}^{(2)} = \{((f_1,g_1),(f_2,g_2)) \in (\mathfrak{m}_x \times G(1)) : f_2 = f_1g_1\}\] are composable; and the product is defined as $(f_1,g_1)(f_2,g_2) = (f_1,g_1g_2)$ is continuous. The inverse is given as $(f,g)^{-1} = (fg,g^{-1})$ is also continuous. The space of units or objects is $\mathcal{G}^o = \mathfrak{m}_x \times \{e\}$, where $e \in G(1)$ is the identity element. The space of objects is therefore identified with the maximal ideals $\mathfrak{m}_x$ of $C(X)$. The groupoid is therefore trivialized on $X$ by the correspondence $\mathfrak{m}_x \leftrightarrow x$. The source and target maps defined as \[s(f,g) = f, \; t(f,g) = fg,\] are continuous surjective maps (local homeomorphisms). Thus, $\mathcal{G} = \mathfrak{m}_x \rtimes G(1)$ is an (analytic) \emph{etale groupoid} or a topological groupoid since $G(1)$ is locally compact group.

To define the groupoid convolution algebra, we define the one-cocycle by the homomorphism $\rho : \mathcal{G} \to \mathbb{T} $ given by $\rho(f,g_1g_2) = \rho(f,g_1)\rho(fg_1, g_2)$, where \[\rho(f,g_1) = \frac{|g_1|}{|f| + |g_1|}, \rho(fg_1,g_2) = \frac{|g_2|}{|fg_1| + |g_2|}.\]
The groupoid $\mathcal{G} \rightrightarrows \mathcal{G}^o$ as defined has a Borel structure; and the multiplicative set $\mathcal{G}^{(2)} \subset \mathcal{G} \times \mathcal{G}$ is a Borel set, and both the product and inverse maps are Borel functions. The source and target maps are also Borel functions. In fact, the functions are also analytic. Hence, the groupoid is analytic.

Now, given a fixed object $(f,e) \in \mathcal{G}^o$, the $t$-fibre over the object $t^{-1}(f,e) = \{(fg^{-1},g)\}_{g \in G(1)}$ supports the probability measure $\mu_f \times \nu$. Assuming that $s(f,g_1) = (f,e)$; then for the subalgebra $\mathfrak{A}$ generated by $(X - Z[\mathfrak{m}_x]$, the map \[A \mapsto \int \chi_A((f,g_1)(fg_1,g_2))d(\mu_f \times \nu)(fg_1,g_2)\] defines a probability measure $(f,g_1)\cdot (\mu_{f} \times \nu) = \mu_{fg_1} \times \nu$ on the $t$-fibre $t^{-1}(\{t(f,g_1)\}) = t^{-1}(\{(fg_1,e)\})$, where the term $(f,g_1)(fg_1,g_2)$ is our one-cocycle. With this formulation, the family of measures $\{(\mu_f \times \nu)\}_{f \in \mathfrak{m}_x}$ has its support on the $t$-fibres $t^{-1}(\{s(f,g_1)\})$, which makes the product $(f,g_1)(fg_1,g_2)$ defined for $(\mu_f \times \nu)$-almost for all $(fg_1,g_2)$.

Next, we show that the family of measures $(\mu_f \times \nu)_{f \in \mathfrak{m}_x}$ is symmetric and translation invariant. By definition, the Dirac measure $\delta_x$ is a probability measure since $\delta_x(X) = 1$. Also, every measure class $C$ is shown to be symmetric, since a Borel measure can be normalized on its conull set. Thus, every class of measures $\mathcal{C}$ in $\mathcal{M}(X)$ is symmetric. Given a measure $(\mu_f \times \nu)^{(f,e)}$ supported by the $t$-fibre $t^{-1}(f,e)$, it is translated by $(f,g_1)$ as follows \[(f,g_1)\cdot (\mu_f \times \nu) = (\mu_{fg_1} \times \nu)^{(fg_1,e)}; \] where $s(f,g_1) = (f,e); t(f,g_1) = (fg_1,e) = s(fg_1,g_2)$. Thus, the family is translation-invariant since $(\mu_f \times \nu) \sim (\mu_{fg_1} \times \nu)$.

From the definition of the measure classes, a probability measure $\mu \in C$ is both $s,t : \mathcal{G} \to \mathcal{G}^o$ decomposable; which is seen in the fact that \[\mu = \int \mu_fdt_*\mu((f,e)) = \int \mu_fd\mu(t(fg,g^{-1})).\] This implies that $(f,e)  \in U \subset \mathcal{G}^o$-a $\delta_x$-conull set and $t(fg,g^{-1}) = (f,e) \in U$. So $s(f,g) = (f,e) \in U$, and $(f,g)(\mu_f \times \nu) = (\mu_{fg} \times \nu)$. Hence, $(\mu_f \times \nu)$ is (left) quasi-invariant.

A symmetric measure class $\mathcal{C}$ is called invariant if a quasi-invariant probability belongs to $\mathcal{C}$ according to Hahn \cite{Hahn78}. Hence, the measure $(\mu_f \times \nu)$ belonging to the family $([\mu_f] \times \nu)_{f \in \mathfrak{m}_x}$ is symmetric and quasi-invariant. The translation-invariant family of measures is the Haar system  of measure for the groupoid $\mathcal{G}\rightrightarrows \mathcal{G}^o$. Subsequently, the pair $(\mathcal{G}, ([\mu_f] \times \nu)_{f \in \mathfrak{m}_x})$ constitute an \emph{etale groupoid} or a topological groupoid by the action of the locally compact group $G(1)$; but a \emph{measure groupoid} by its ergodic action.
\end{proof}
The Haar system of measures $([\mu_f] \times \nu)_{f \in \mathfrak{m}_x}$ is not unique since each $f \in \mathfrak{m}_x$ defines a Haar system and a measure groupoid. Also, since we have as many maximal ideal as there are points in $X$, the pairs $(\mathcal{G}_x,([\mu_f] \times \nu)_{f \in \mathfrak{m}_x})$ aggregate to a bundle of measure groupoids. Let $\mathcal{C} = ([\mu_f] \times \nu)$ for some$f \in \mathfrak{m}_x$ be a class of measures on the $\mathcal{G}$, then from the definition of a bundle groupoid, we have a bundle groupoid $\mathcal{G} = \underset{x \in X}\bigcup (\mathcal{G}_x, C)$, which is also a measure groupoid since it is a union of measure groupoids. The composable is given as $\mathcal{G}^{(2)} = \underset{i \in I}\bigcup \mathcal{G}_x^{(2)}$, where the product of $((f,g_1)(fg_1,g_2) \in \mathcal{G}_x^{(2)}$ is as in $\mathcal{G}_x$ for each $x \in X$. This groupoid structure now helps us to redefine the action of $C(X)$ on the generalized space $\mathcal{M}(X)$.

\section{Action of Measure Groupoid}
The characterization of measure groupoid (virtual group) using the idea of action groupoid, according to Ramsay \cite{Ramsay71}, helps us to study groupoids from the viewpoint of group representation. From his definition of action of a groupoid on a set, we define the action of the measure groupoid $\mathcal{G} = \mathfrak{m}_x \rtimes G(1)$ on the generalized space $\mathcal{M}(X)$. Instead of stabilizers of points in $X$, we are interested in stabilizers of measure classes or generalized subsets $\{\mathcal{C} \in \mathscr{C}$ defined by $f \in \mathfrak{m}_x\}$. This role is played by the Borel group $G(1)$ through its Borel action on $\mathfrak{m}_x$ which preserve the zero sets $Z[\mathfrak{m}_x]$. Therefore, $G(1)$ is the stabilizer of each measure class $\mathcal{C}$.

Further, following from the definition of the homotheties $T_{x,r} \in Aut(X,\mu)$ using $G(1)$ and $\mathfrak{m}_x$, we will formulate the $Aut(X,\mu)$-action on $C(X)$ in terms of the action of the arrows $\mathcal{G}\rightrightarrows$ of the measure groupoid on the open subspaces $\{\mathcal{G}((f,e),-) : (f,e) \in \mathfrak{m}_x\}$, and on its space of units $\mathfrak{m}_x \times \{e\}$; and the $Aut(X,\mu)$-action on $\mathcal{M}(X)$ in terms of the groupoid action on $\mathcal{M}(X)$. Before we give another main result of this paper we recall the definitions of a section and a bisection, as applicable to our context and to clarify the notion of a slice in the generalized space.
\begin{defx}
A section or a canonical form for the generalized space $\mathcal{M}(X)$ is a connected, closed, regularly embedded subspace $\Sigma$ of the generalized space $\mathcal{M}(X)$ that meets all orbits orthogonally. (See. \cite{BerndtConsoleOlmos})
\end{defx}
\begin{defx}\cite{Mackenzie2005}
A local section (cross section) $\sigma : M \to P$ is a continuous right inverse of the projection $\pi : P \to M$ of a fibre bundle. It identifies the base space $M$ with $\sigma(M) \subset P$.
\end{defx}
\begin{defx}\cite{Mackenzie2005}
A bisection of a groupoid $\mathcal{G} \rightrightarrows \mathcal{G}^o$ is a continuous (smooth) map $\sigma : \mathcal{G}^o \to \mathcal{G}$ which is inverse to the source map $s : \mathcal{G} \to \mathcal{G}^o$ and is such that $t \circ \sigma : \mathcal{G}^o \to \mathcal{G}^o$ is a homeomorphism.
\end{defx}
\begin{prop}
Let $\mathcal{G} \rightrightarrows \mathcal{G}^o$ be the action groupoid $\mathfrak{m}_x \rtimes G(1)$ resulting from the $G(1)$-action on $\mathfrak{m}_x$, then $\mathcal{G}\rightrightarrows$ has a proper action on the space of units $\mathfrak{m}_x \times \{e\}$ and on the open subspace $\mathcal{G}((f,e),-), (f,e) \in \mathfrak{m}_x \times \{e\}$.
\end{prop}
\begin{proof}
Let $G(1) \times \mathfrak{m}_x \to \mathfrak{m}_x$ be the action of $G(1)$. The action is transitive, for given any two functions $f_1, f_2 \in \mathfrak{m}_x$ which do not vanish on the open set $U = U_{f_1}\cap U_{f_2} \subset X$; we have $f_1(y),f_2(y) \in \R$ for any $y \in U$, and there exists $a \in \R$ such that $af_1(y) = f_2(y) \; \implies \; g(y)f_1(y) = f_2(y) \; \implies \; gf_1 = f_2$, where $g(y) = a$ for some $g \in G(1)$.

The transitivity of the $G(1)$-action, according to \cite{Mackenzie2005}, Proposition 1.3.3, means that the map $t_{(f,e)} : (f,g) \mapsto (fg,e); \mathcal{G}((f,e),-) \to \mathfrak{m}_x \times \{e\}$ is a surjective submersion; and the division map $D_{(f,e)} :  \mathcal{G}((f,e),-) \times \mathcal{G}((f,e),-) \to \mathcal{G}, \; ((f,g_2),(f,g_1)) \mapsto (f,g_2)(f,g_1)^{-1} = (f{g_2g_1^{-1}}, g_2g_1^{-1})$ is a groupoid morphism over $t : \mathcal{G} \to \mathfrak{m}_x \times \{e\}$. Thus, there exists an open cover $\{\mathcal{V}_i\}$ of the space of units $\mathfrak{m}_x \times \{e\}$ by the images of bisections $\sigma ((f,e)) = \mathcal{G}((f,e),-)$ (which in this case are local sections over the source map) such that $t \circ \sigma((f,e)) = t|_{(f,e)} = t_{(f,e)} : \mathcal{G}((f,e),-) \to \mathfrak{m}_x \times \{e\}$ is a homeomorphism. Hence, the ergodic action of the measure groupoid $\mathcal{G} = \mathfrak{m}_x \times G(1)$ on its base space $\mathcal{G}^o = \mathfrak{m}_x \times \{e\}$ is proper because it is given by the map $t,s : \mathcal{G}\rightrightarrows \to \mathcal{G}^o \times \mathcal{G}^o, (f,g)  \mapsto (s(f,g),t(f,g))$ which is a proper map.

Likewise, the morphism $D_{(f,e)} :  \mathcal{G}((f,e),-) \times \mathcal{G}((f,e),-) \to \mathcal{G}$ defined over the target map $t$ is a proper map taking a compact space $\mathcal{G}(s(f,g),s(f,g))$ to a compact space $\mathcal{G}(s(f,g),t(f,g))$ for all $\mathcal{G}(s(f,g),-)$-composable arrows. The compactness of the closed embedded subspaces $\mathcal{G}((f,e),t(f,g))$ is evident in the fact that the ergodic actions of the measure groupoid $\mathcal{G} = \mathfrak{m}_x \times G(1)$ preserve measure classes; which entails metric invariance or transitivity. Thus, that the measure groupoid action is proper follows from the compactness of the space of arrows $\mathcal{G} = \mathfrak{m}_x \rtimes G(1)$, which correspond to the compactness of $Aut(X,\mu) \subset \mathcal{U}(L^2(X,\mu))$.
\end{proof}
The next result shows that the measure groupoid $\mathcal{G}$ has a true action on the generalized space $\mathcal{M}(X)$, which preserves the $\mathfrak{m}_x$-partition of the generalized space. First, we give the definitions of a true groupoid action and true homomorphism.
\begin{defx}\cite{Ramsay71}
A left action of a groupoid on a set $\mathcal{G} \times S \to S$, is said to be true if for a pair of arrows $g,h \in \mathcal{G}$, if $u \in Dom(h)$ and $h(u) \in Dom(g)$ implies that $gh$ is defined in $\mathcal{G}$. A true right action is similarly defined.
\end{defx}
\begin{defx}
A homomorphism $\phi : \mathcal{G}_1 \to \mathcal{G}_2$ (defining a representation) of groupoids is true if $gh \; \iff \; \phi(g)\phi(h)$. Thus, given a true homomorphism of groupoids $\phi : \mathcal{G}_1 \to \mathcal{G}_2$, then $\phi(\mathcal{G}_1)$ is a subgroupoid of $\mathcal{G}_2$.
\end{defx}
\begin{thm}
Given the measure groupoid $(\mathcal{G}, \mathcal{C})$ or $\mathcal{G}\rightrightarrows \mathcal{G}^o$ trivialized on $X \simeq \mathcal{U}$, the map $\phi : F = \mathcal{G} \times \mathcal{M}^1(X) \to \mathcal{M}^1(X)$ is a true action of the measure groupoid $(\mathcal{G},C)$ on $\mathcal{M}^1(X) = \{D((f,g)) : (f,g) \in \mathcal{G}\}$. Thus, we have $\phi : F = \mathcal{G} \times \mathcal{M}^1(X) \to \mathcal{M}^1(X)$.
\end{thm}
\begin{proof}
From the definition of the groupoid above, the set $\mathfrak{m}_x \times \{e\} = \{(f,e) : f \in \mathfrak{m}_x\}$ is the space of units in $\mathcal{G}$. By the $\mathfrak{m}_x$-decomposition, there is a homomorphism \[\phi : \mathcal{G}^o \to Inj(\mathcal{M}(X))|_{\mathfrak{m}_x}, \phi((f,e))(\mu_f) = (f,e)\mu_f \to \mu_f\] for any unit $(f,e) \in \mathcal{G}^o$.

Let $\mathcal{M}^1(X) \simeq \mathbb{T}$ be the normalized space of Radon measures on $X$. Then $\phi$ is a homomorphism between the object space $\mathcal{G}^o = \mathfrak{m}_x \times \{e\}$ and the set of injections $Inj(\mathcal{M}^1(X))$, preserving null sets of measures; and subsequently, the $\mathfrak{m}_x$-partition of $\mathcal{M}^1(X))$.

The map $\phi : \mathcal{G} \rightrightarrows \mathcal{G}^o \to Inj(\mathcal{M}(X))$ is a homomorphism for it preserves the continuous partial product $(f_1,g_1)(f_2,g_2) = (f_1,g_1g_2)$, where $f_2 = f_1g_1$, on $\mathcal{G}$. This is shown as follows; \[\phi((f_1,g_1)(f_2,g_2))(\mu) = \phi((f_1,g_1))\phi((f_2,g_2))(\mu) \] \[= \phi((f_2,g_2))(\mu_{f_1g_1}) = \mu_{f_1g_1g_2} = \mu_{f_2g_2}.\] Further, as given above, $\phi((f,e))\mu_f = (f,e)\mu_f \to \mu_f$, for any $(f,e) \in \mathcal{G}^o$. Thus, the set $\{(f,e) : f \in \mathfrak{m}_x\}$ maps to the set of units in $Inj(\mathcal{M}(X))$. Hence, $\phi$ defines the $\mathcal{G}$-action preserving the partition determined by the $\mathfrak{m}_x$-decomposition. \\
Finally, the action of the measure groupoid $(\mathcal{G},C)$ on the $\mathcal{M}^1(X)$  is true since the fact that $\mu \in D((f_1,g_1))$ and $(f_1,g_1)\mu \in D(f_2,g_2)$ implies the definition of the product $(f_1,g_1)(f_2,g_2)$.

Thus, $\phi : \mathcal{G} \times \mathcal{M}^1(X) \to \mathcal{M}^1(X)$ is a true action of the measure groupoid $(\mathcal{G},C)$ on $\mathcal{M}^1(X) = \{D((f,g)) : (f,g) \in \mathcal{G}\}$ and $\phi((f_1,g_1)), \phi((f_2,g_2)) \in Inj(\mathcal{M}(X))$.
\end{proof}
The true action of a groupoid corresponds to the Cayley's Theorem in the context of groupoids, according to Ramsay \cite{Ramsay71}. The Cayley's Theorem for groupoid relates the transitivity of $\mathcal{G}$ on the $t$-fibres to the \emph{trueness} of the $\mathcal{G}$-action. Thus, the $t$-fibre over any base or generalized point $\delta_x \in \mathcal{M}(X)$, that is, the measure classes $\mathscr{C} \simeq \{t^{-1}(\phi^{-1}(\mu)) = \mathcal{G}^{(f,e)}: (f,e) \in \mathcal{G}^o\}$, are partitions. So that $\phi : \mathcal{G} \to Inj(\mathcal{M}(X), \mathscr{C})$ is a true homomorphism. This makes $\mathcal{G}$ a groupoid of bundle bijections. This is made more obvious in the following slice analysis.

\section{Slice Analysis}
The equivariance  of the action of $Aut(X,\mu)$ on $C(X)$ and $\mathcal{M}(X)$ is shown in the proof of trueness of the measure groupoid $\mathcal{G}$-action on the generalized space $\mathcal{M}(X)$. Thus, given a Radon measure $\mu \in \mathcal{M}(X)$, with $U_\mu$ as the neighbourhood, the $Aut(X,\mu)$-action is easily seen to be proper at $\mu$ by the fact that the set $\{\varphi \in Aut(X,\mu) : \varphi_*(U_\mu) \cap U_\mu \not = \emptyset \}$ is relatively compact in $Aut(X,\mu)$. This is because of the ergodicity of $Aut(X,\mu)$-action, which implies that $\mu \in \mathcal{M}(X)$ (or its measure class $[\mu]$) is $Aut(X,\mu)$-invariant.
Using these, we now extend the result of \cite{OkekeEgwe2018}, Theorem 2.7 to the case of the action of the measure groupoid $(\mathcal{G},\mathcal{C})$ on the generalized space $\mathcal{M}(X)$.

\begin{thm}
Given the measure groupoid $(\mathcal{G},\mathcal{C})$ action on $\mathcal{M}(X)$. The orbit $O_{\mu}$ of the action of the arrows $\mathcal{G}\rightrightarrows$ passing through a point $\mu \in \mathcal{M}(X)$ has a slice at each point; hence, is a cohomogeneity-one $\mathcal{G}$-space.
\end{thm}
\begin{proof}
The result follows from \cite{OkekeEgwe2018} once we understand the affine and geometric nature of the measure groupoid action on the tangent bundle $T(\mathcal{M}(X))$ over the compact space $\mathcal{U} = \{\delta_x : x \in X\}$ of Dirac measures on $X$. We outline the geometric features of the $\mathcal{G}$-space as follows. By definition $Aut(X,\mu) \subseteq Inj(\mathcal{M}(X))$. Thus, the homomorphism $\phi : \mathcal{G} \to Inj(\mathcal{M}(X): \mathscr{C})$ makes $Aut(X,\mu)$-action equivalent to $\mathcal{G}$-action on $\mathcal{M}(X)$. This means that the measure classes are preserved; and hence the geometric independence of the measures in the space $\mathcal{M}(X)$.

By definition the generalized points $\mathcal{U} = \{\delta_x : x \in X\} \subset \mathcal{M}(X)$ are geometrically independent. So, we have $Aut(X,\mu) \times \mathcal{U} \to \mathcal{U}$ as a restriction of the $Aut(X,\mu)$-action on the generalized space $\mathcal{M}(X)$. Thus, for any $\varphi \in Aut(X,\mu)$, the restriction of $\varphi_*$ to the subset $\mathcal{U}$ gives a translation or transformation $\varphi_* : \delta_x \mapsto \delta_{\varphi(x)}$ of the invariant ergodic measure. $\varphi_*$ continuously and affinely extends $\mathcal{U}$ to the generalized space $\mathcal{M}(X)$, giving rise to a bundle of tangent spaces of measures $Tan(\delta_x,x) \simeq \mathcal{M}(X)\setminus \{0\}$.

Since by the local triviality of the measure groupoid $\mathcal{G} = \mathfrak{m}_x \rtimes G(1)$ there is a vertex bundle at each $f \in \mathfrak{m}_x = \ker \delta_x$ such that $\mathcal{G}((f,e),-)(\mathfrak{m}_x, \mathcal{G}((f,e),(f,e)),t_{(f,e)})$ are principal bundles. Given any other $(h,e) \in \mathfrak{m}_x \times \{e\}$, we have $\varphi_* \in \mathcal{G}((f,e),(h,e)) : \varphi \in Aut(X,\mu)$ defining a bundle translation \[ R_{\varphi_*} : \mathcal{G}((f,e),-)[\mathfrak{m}_x \times \{e\}, \mathcal{G}((f,e),(f,e))] \to \mathcal{G}((h,e),-) [\mathfrak{m}_x \times \{e\}, \mathcal{G}((h,e),(h,e))], \] where $h = \varphi_*(f) = f \circ \varphi^{-1}$ is equivalent to $\nu = \varphi_*(\mu) = \mu \circ \varphi$. This is an isomorphism of principal bundles over $\mathfrak{m}_x \times \{e\}$ and \[I_{\varphi_*} : \mathcal{G}((f,e),(f,e)) \to \mathcal{G}((h,e),(h,e)), \] is an isomorphism of affine (Lie) groups induced by $\varphi_* \in \mathcal{G}((f,e),(h,e))$.

By the identification of the compact regular space $X$ with $\mathcal{U}$, the pair of homeomorphisms $(\varphi,\varphi_*)$ defines a bundle translation, which the homomorphism $\phi : \mathcal{G} \to Inj(\mathcal{M}(X))$ makes a groupoid translation by trivializing the source map $s$ as follows $t \circ \varphi_* = \varphi \circ t$ and $s \circ \varphi_* = s$, where $\phi((f,g)) = \varphi_*$, for some $(f,g) \in \mathcal{G}$.
\begin{center}
\begin{tikzpicture}
\draw[->] (0.3,0) -- (4.5,0) node[right] {$\mathcal{U}$};
\draw[-] (2.5,0) node[above] {$\varphi$};
\draw[<-] (5,0.3) -- (5,1.5) node[above] {$T_{\delta_{\varphi(x)}}\mathcal{G}(\delta_{\varphi(x)})$};
\draw[-] (5,0.9) node[right] {$t$};
\draw[-] (0,1.5) node[above] {$T_{\delta_x}\mathcal{G}(\delta_x)$};
\draw[->] (0,1.5) -- (0,0.3) node[below] {$\mathcal{U}$};
\draw[-] (0,0.9) node[left] {$t$};
\draw[->] (1,1.7) -- (3.4,1.7);
\draw[-] (1.7,1.7) node[above] {$\varphi_*$};
\end{tikzpicture}
\end{center}
This defines a homeomorphism of $t$-fibres (tangent measures): $\varphi_* : t^{-1}(\delta_x) \to t^{-1}(\delta_{\varphi(x)})$ preserving the measure classes ($\mathscr{C}$-partition). For any other non-ergodic measure $\lambda \in t^{-1}(\delta_x))$, we have $\varphi_*(\lambda) = \lambda \circ \varphi \in t^{-1}(\delta_x)$. From which we see that the pair of morphisms $(\varphi,\varphi_*)$ corresponds to a bisection $\sigma : \mathcal{G}^o \to \mathcal{G}, \; (f,e) \mapsto \mathcal{G}((f,e), -)$ of the measure groupoid, and $t \circ \sigma : \mathcal{G}^o \to \mathcal{G}^o, \; (f,e) \mapsto (f\circ \varphi^{-1}, e)$ is a homeomorphism. Since the homomorphism $\phi : \mathcal{G} \to Inj(\mathcal{M}(X))$ is true, the $\mathcal{G}$-action on the $t$-fibres $t^{-1}((f,e)) = \mathcal{G}(-,(f,e))$ is transitive. Hence, $\mathcal{G} \simeq Inj(\mathcal{G}, \{t^{-1}((f,e))\}_{(f,e)\in \mathcal{G}^o}) \simeq Inj(\mathcal{G},\mathcal{C})$.

Now, according to \cite{Mackenzie2005}, the image of a bisection $\sigma((f,e)) = \mathcal{G}((f,e),-)$ is a closed embedded subspace of $\mathcal{G}$, and any such embedded subspace is the image of a unique bisection on which $s,t$ are homeomorphisms. Therefore, the measure classes, defined by null sets, assures the existence of slices at every point $\mu \in \mathcal{M}(X)$. Hence, the existence of slice at each point of $\mathcal{M}(X)$ follows from the fact that each of the subsets $K_i = \{\delta_y : y \in F_i\}$ is geometrically independent in $\mathcal{M}(X)$. Thus, the geometrical independence of these sets constituting the $z$-filter or ultrafilter defines a slice at $\mu \in Tan(\delta_x,x) \simeq \mathcal{M}(X)\setminus\{0\}$. This implies existence of a section $\{\nu_i\} = \Sigma \subset \mathcal{M}^1(X)$ such that action $\mathcal{G}((f,e),(f,e)) \times \Sigma \to \Sigma$ simplifies the original $\mathcal{G}$-action, such that, $\mathcal{G}((f,e),(f,e)) \times \Sigma  \simeq \mathcal{G} \times \mathcal{M}^1(X)$.

The generalized subsets $\mathscr{C}$ (or the measure classes) could be considered as constituting the principal orbit through the set of Dirac measures $\mathcal{U} = \{\delta_x : x \in X\}$. For any other non-ergodic Radon measures taken from any of the classes $\{\mu_f\}_{f \in \mathfrak{m}_x}$ at $x \in X$, which are related to the Dirac measure $\delta_x$, the orbit (tangent space) is smaller. Thus, the dimension of the principal orbit of the $\mathcal{G}$-action is one less than the dimension of $\mathcal{M}(X)$; Hence, $\mathcal{M}(X)$ is a cohomogeneity-one $\mathcal{G}$-space. Thus, the section $\Sigma$ orthogonally meets all the orbits of the $\mathcal{G}$-action which are the measure classes.
\end{proof}

\section{conclusion}
Central to this analysis is the polar action of the automorphism group $G(1) \subset Aut(X,\mu)$ on the generalized space $\mathcal{M}(X)$, which was defined using the commutative algebra $C(X)$. Ergodic requirements made it into the dynamical system of action groupoids $G(1)\times \mathfrak{m}_x \rightrightarrows \mathfrak{m}_x$. This proper isometric action is realized as a true action of the measure groupoid $\mathcal{G}$ on the tangent bundle of measures over the compact subset $\mathcal{U} = \{\delta_x : x \in X\}$; and the definition of bisections shows that the $t$-fibres $t^{-1}(\varphi_*(\delta_x)) = \mathcal{G}(-,\delta_{\varphi(x)})$ have the transitive $\mathcal{G}$-action. Hence, $\phi : \mathcal{G} \to Inj(\mathcal{M}(X))$ is a groupoid representation.

The orbit of the measure groupoid action is an immersed submanifold or section $\Sigma$, which meets every measure class orthogonally. While in the case of a Lie groupoid, the section $\Sigma$ is defined by a horizontal distribution $\mathcal{H}$ whose integral manifold gives the foliation of the groupoid. Here it is defined by sequence/net of $\varphi$-invariant measures related to the null hyperplane $\mathfrak{m}_x$ making up the closed set $\mathcal{M}(X)_\varphi$, which net converging to $\delta_x$ defining a dynamical representation of each point $x \in X$ on the tangent spaces of measures $Tan(\delta_x, x)$ constituting the measure bundle $T(\mathcal{M}(X))$. 

Thus, this instantiates the classical Chevalley restriction Theorem for the adjoint action of a compact group $G$ as here applied to groupoid. The ergodic groupoid $G(1) \subset Aut(X,\mu)$ is compact, a $G(1)$-invariant function on $\mathcal{M}(X)$ is integrable if and only if its restriction to the section $\Sigma$ is integrable, and invariant under the stabilizer of the section $\Sigma$. \cite{GroveZiller2012}. The orbits of the ergodic $G(1)$-action are the measure classes; and the ergodic elements constitute the boundary of the closed subspace $\mathcal{M}(X)_\varphi$.

\bibliographystyle{amsplain}

\begin{thebibliography}{99}
\bibitem{AtiyahDonald69} Atiyah, M.\ F.\ and Macdonald, I.\ G.\ \emph{Introduction to Commutative Algebra}, Addison-Wesley Publishing Co., 1969.
\bibitem{BerndtConsoleOlmos} Berndt, J.\; Console, S.\ and Olmos, C.\ \emph{Submanifolds and Holonomy}, Chapman and Hall/CRC, 2003.
\bibitem{EinsWard2011} Einsiedler, M.\ Ward, T.\ \emph{Ergodic Theory}, Springer, 2011.
\bibitem{Gillman1960} Gillman, Leonard. Jerison, Meyer. \emph{Rings of Continuous Functions}, V. Dan Nostrand Company, Inc., 1960.
\bibitem{GroveZiller2012} Grove, K.\  Ziller, W.\ \emph{Polar Manifolds and Actions}, arXiv:1208.0976v2 [math.DG] 8 Sept. 2012.
\bibitem{Hahn78} Hahn, Peter. \emph{Haar Measure for Meaure Groupoids}. Transactions of the American Mathematical Society, vol. 242, August 1978, pp. 1-33.
\bibitem{HasselblattKatok} Hasselblatt, Boris. and Katok, Anatole. \emph{A First Course in Dynamics}, Cambridge Unversity Press, 2003.
\bibitem{Jaroslavetal2010} Luke$\breve{s}$, J.\ Maly, J.\ Netuka, I.\ Spurny J.\ \emph{Integral Representation Theory}. Walter de Gruyter, 2010.
\bibitem{Kechris2010} Kechris, S.\ A.\ \emph{Global Aspects of Ergodic Group Actions}. Americal Mathematical Society, Mathematical Surveys and Monographs, volume 160.
\bibitem{Ma2002} Ma, Tsoy-Wo. \emph{Banach-Hilbert Spaces, Vector Measures and Group Representations}. World Scientific, 2002.
\bibitem{Mackey57} Mackey, G.\ W.\ \emph{Borel Structure in Groups and their Duals}, Trans. AMS. 85, 134 (1957).
\bibitem{Mackey66} Mackey, G.\ W.\ \emph{Ergodic Theory and Virtual Groups}, Mathematische Annalen | Journal 349 page(s).
\bibitem{Mackenzie2005} Mackenzie, K.\ C.\ \emph{General Theory of Lie Groupoids and Lie Algebroids}, Cambridge Univ. Press, 2005.
\bibitem{OkekeEgwe2018} Okeke, O.\ N.\ and Egwe, M.\ E.\ \emph{Groupoid in the Analysis of Cohomogeneity One $G$-spaces}, Far East Journal of Mathematics (FJMS). Vol. 103, No.12, 2018, pp. 1903-1920.
\bibitem{OkekeEgwe2017} Okeke, O.\ N.\ and Egwe, M.\ E.\ \emph{Cohomogical Stability and Harmonic Analysis of Discrete Dynamical Systems of Compact Connected Lie Groups: A Survey}, a paper submitted to Journal of the Nigerian Association of Mathematical Physics, September, 2020.
\bibitem{O'Neil95} O'Neil T.\ \emph{A Measure with a Large Set of Tangent Measures}, Proceedings of the American Mathematical Society, Vol. 123, Number 7, 1995.
\bibitem{Ramsay71} Ramsay, Arlan. \emph{Virtual Groups and Group Actions}. Advances in Mathematics 6, 2530-322 (1971).
\bibitem{Sahlsten2014} Sahlsten T.\ \emph{Tangent Measures of Typical Measures}, ArXiv.1203.4221v3 [math.CA]1 Dec 2014.
\bibitem{Semadeni65} Semadeni, Zbigniew. \emph{Spaces of Continuous Functions on Compact Sets}, Advances in Mathematics, Vol. 1, Issue 3, 1965 (319-382).
\bibitem{Strocchi2008} Strocchi, F. \emph{Symmetry Breaking}. 2nd Edition, Springer 2008.
\bibitem{Varadarajan62} Varadarajan, V.\ S.\ \emph{Groups of Automorphisms of Borel Spaces}, Journal of AMS, 1962.

\end{thebibliography}

\providecommand{\bysame}{\leavevmode\hbox to3em{\hrulefill}\thinspace}

\end{document}